%% file: ricks-product-splittings.tex
\documentclass{amsart}

\input{artcfg}

\usepackage{tikz}

\title{Detecting product splittings of CAT(0) spaces}

\author{Russell Ricks}
\address{Binghamton University, Binghamton, New York, USA}
\email{ricks@math.binghamton.edu}

\thanks{The author would like to thank Ralf Spatzier for his support and many helpful discussions, Eric Swenson for his helpful comments, Alexander Lytchak for suggesting characterization \itemrefstar{me6} in \thmrefstar{equivalences}, and both the University of Michigan and Binghamton University, where the work was completed.
A couple anonymous referees also offered helpful suggestions on a previous version of the paper, including the short proof of \thmrefstar{main radius rigidity} given here.%
}
\thanks{
This material is based upon work supported by the National Science Foundation under Grant Number NSF 1045119.}

\date{\today}

\begin{document}

\maketitle

\begin{abstract}
Let $X$ be a proper CAT(0) space and $G$ a group of isometries of $X$ acting cocompactly without fixed point at infinity.
We prove that if $\partial X$ contains an invariant subset of circumradius $\pi/2$, then $X$ contains a quasi-dense, closed convex subspace that splits as a product.

Adding the assumption that the $G$-action on $X$ is properly discontinuous, we give more conditions that are equivalent to a product splitting.
In particular, this occurs if $\partial X$ contains a proper nonempty, closed, invariant, $\pi$-convex set in $\partial X$; or if some nonempty closed, invariant set in $\partial X$ intersects every round sphere $K \subset \partial X$ inside a proper subsphere of $K$.
\end{abstract}

\section{Introduction}

An important insight of geometric group theory is that a properly discontinuous, cocompact, isometric group action (that is, a geometric action) on a space $X$ often reveals additional information about $X$ itself through the dynamics of the induced action on the boundary of $X$.
In this paper we study product splittings of CAT($0$) spaces.
These spaces are geodesic metric spaces satisfying a nonpositive curvature condition that generalizes nonpositive curvature in Riemannian geometry; this allows one to study nonpositive curvature outside the smooth manifold setting.
We give some ways the product structure of a CAT($0$) space can be detected from the dynamics of the boundary action.

Let $X$ be a proper CAT($0$) space under a geometric action.
Papasoglu and Swenson \cite{ps} describe a way to detect a product splitting with a line, based on work of Ruane \cite{ruane01} and Swenson \cite{swenson99}:
If $\bd X$ has a fixed point, then $\bd X$ must be a suspension.
It is well-known \cite{bridson} that when $X$ is geodesically complete, if $\bd X$ is a suspension then $X$ must split as a product $X = Y \times \R$ for some proper CAT($0$) space $Y$.
This conclusion holds, up to passing to a quasi-dense, closed convex subspace, without the hypothesis of geodesic completeness.

Since every closed, invariant set in $\bd X$ of circumradius $< \frac{\pi}{2}$ has unique circumcenter, fixed by the action, the presence of an invariant set in $\bd X$ of radius $< \frac{\pi}{2}$ reveals a product splitting with $\R$.

For general product splittings $X = Y \times Z$ (allowing $Z \neq \R$), one finds that the boundary splits as a spherical join $\bd X = \bd Y * \bd Z$.
Both subsets $\bd Y \subset \bd X$ and $\bd Z \subset \bd X$ are closed, have radius exactly $\frac{\pi}{2}$, and (up to passing to a subgroup of finite index) are invariant under the group action \cite{fl}.
Also, both sets are $\pi$-convex.
Thus it is natural to ask if the presence of such a set reveals a product splitting.
More precisely, consider the following two questions, where $X$ is a proper \textnormal{CAT($0$)} space under a geometric action.
\begin{enumerate}
\item
Can one detect a product splitting of $X$ by finding a nonempty invariant set in $\bd X$ that has circumradius $\frac{\pi}{2}$?
\item
Can one detect a product splitting of $X$ by finding a proper nonempty, closed, invariant, $\pi$-convex set in $\bd X$?
\end{enumerate}
In this paper we show that the answer to both questions is yes.
In fact, the first result holds without requiring the group action to be properly discontinuous.

\begin{maintheorem} [\thmref{minimal radius rigidity}]	\label{main radius rigidity}
Let $X$ be a proper \textnormal{CAT($0$)} space with finite-dimensional Tits boundary, and let $G \le \Isom(X)$ be geometrically dense.
If $\bd X$ has a nonempty $G$-invariant subset of circumradius $\le \frac{\pi}{2}$
then $X$ splits as a nontrivial product. \end{maintheorem}

This gives us the following corollary.

\begin{maincorollary} [\corref{radius rigidity}]	\label{main radius rigidity corollary}
Let $X$ be a proper \textnormal{CAT($0$)} space and let $G \le \Isom(X)$ act cocompactly on $X$ without fixed point at infinity.
If $\bd X$ has a nonempty $G$-invariant subset of circumradius $\le \frac{\pi}{2}$
then a quasi-dense, closed convex subspace $X' \subseteq X$ splits as a nontrivial product. \end{maincorollary}

We remark that
if $G \le \Isom(X)$ is a unimodular closed subgroup (for instance, if $G$ is discrete), and $G$ acts cocompactly, then the presence of fixed points in $\bd X$ already indicates a nontrivial Euclidean factor of a quasi-dense, closed convex $X' \subseteq X$ \cite[Theorem M]{cm-fpa}.

In the case that $G = \Gamma$ is discrete, we obtain the following, stronger form of \corref{main radius rigidity corollary}, providing additional ways to detect a general product splitting in $X$.

\begin{maintheorem}					\label{equivalences}
Let $\G$ be a group acting properly discontinuously, cocompactly, and by isometries on a proper \textnormal{CAT($0$)} space $X$ with $\abs{\bd X} \ge 3$.  The following are equivalent.
\begin{enumerate}
\item
\label{me1}
For some finite index subgroup $\G_0$ of $\G$, there is a nonempty closed, $\G_0$-invariant set $A \subset \bd X$ such that for every round sphere $K \s \bd X$, $A \cap K$ lies on a proper subsphere of $K$.
\item
\label{me2}
For some finite index subgroup $\G_0$ of $\G$, $\bd X$ contains a nonempty, $\G_0$-invariant subset of circumradius $\le \frac{\pi}{2}$.
\item
\label{me6}
For some finite index subgroup $\G_0$ of $\G$, $\bd X$ contains a nonempty, closed, $\pi$-convex, $\G_0$-invariant, proper subset.
\item
\label{me3}
$\bd X$ splits as a nontrivial spherical join.
\item
\label{me4}
A quasi-dense, closed convex set $X' \s X$ splits as a nontrivial product.
\end{enumerate}
Moreover, if $X$ is minimal (for instance, if $X$ is geodesically complete) then each of the above is equivalent to
\begin{enumerate}
\setcounter{enumi}{5}
\item
\label{me5}
$X$ splits as a nontrivial product.
\end{enumerate}
\end{maintheorem}

\section{Centers in $\bd X$}

Let $X$ be a proper CAT($0$) space (\defn{proper} meaning that closed balls in $X$ are compact).
Denote by $\Isom(X)$ the isometry group of $X$.

The CAT($0$) space $X$ has a natural geometric boundary $\bd X$.  It can be obtained by taking all geodesic rays from a fixed basepoint $x_0 \in X$, under the compact-open topology---this is the standard topology on $\bd X$, often called the \defn{cone topology}.  If one takes all geodesic rays and (possibly trivial) geodesic segments from $x_0 \in X$, one obtains the compact Hausdorff space $\cl X = X \cup \bd X$ (compact because $X$ is proper).  One can show that the choice of $x_0$ does not matter---$\cl X$ is the same, up to homeomorphism, independent of $x_0 \in X$.

There is a natural metric on the set $\bd X$, called the \defn{angle metric} and denoted $\angle$, which is obtained by setting $\angle (p,q) = \sup_{x \in X} \angle_x (p,q)$, where $\angle_x (p,q)$ is the angle between $p,q \in \bd X$ as seen from $x \in X$.  The \defn{Tits metric}, denoted $\dT$, is then the path metric associated to $\angle$.  The angle metric $\angle$ on $\bd X$ satisfies $\angle (p,q) = \min \set{\pi, \dT(p,q)}$ for all $p,q \in \bd X$, hence $\angle$ can be recovered as the cutoff metric of the Tits metric.  Thus $\angle$ and $\dT$ induce the same topology on $\bd X$; however, this topology is at least as fine, and generally much finer, than the cone topology.

As $\bd X$ is metrizable under the cone topology, and the angle metric induces a finer topology than the cone topology (in general), there is some ambiguity in referring to $\bd X$.  However, the following simple convention suffices to distinguish the two for our purposes in this paper.

\begin{convention}
In this paper, when we refer to a topology on $\bd X$ (such as describing a subset as closed), we will always mean the cone topology.  Likewise, when we refer to a metric $d$ on $\bd X$, we will always mean the angle metric $d := \angle$. \end{convention}

When it seems prudent to emphasize that we are dealing with $\bd X$ as a metric space, we will refer to it as the \defn{Tits boundary}.  

We mention one more standard, important relationship between the cone topology and the angle metric on $\bd X$:

\begin{fact*}
The angle metric is lower semicontinuous with respect to the cone topology. \end{fact*}

\begin{definition}
Let $X$ be a CAT(0) space and let $G \le \Isom(X)$.  Write $\closedsets$ for the collection of all proper, nonempty, closed subsets of $\bd X$, and write $\CIG$ for the collection of all proper, nonempty, closed, $G$-invariant subsets of $\bd X$. \end{definition}

\begin{definition}
Let $Y$ be any metric space, and let $y \in Y$ and $A \s Y$.
Let $d(y, A) = \inf_{a \in A} d(y, a)$ denote the \defn{metric distance} from $y$ to $A$, and $\Hd(y, A) = \sup_{a \in A} d(y, a)$ the \defn{Hausdorff distance} from $y$ to $A$. \end{definition}

From semicontinuity of the angle metric, we obtain the following observation.

\begin{lemma}						\label{Hausdorff semicontinuity}
Let $A \in \closedsets$.  Then both $d(\x, A) \colon \bd X \to \R$ and $\Hd(\x, A) \colon \bd X \to \R$ are lower semicontinuous. \end{lemma}

\begin{definition}
Let $Y$ be any metric space, and let $A \s Y$.
Let
\[\radius^Y (A) = \inf_{y \in Y} \Hd(y, A)\]
denote the \defn{radius} (or \defn{circumradius}) of $A$ in $Y$, and let
\[\Centers^Y (A) = \setp{y \in Y}{\Hd(y, A) = \radius^Y (A)}\]
denote the set of \defn{centers} (or \defn{circumcenters}) of $A$ in $Y$.
When $Y$ is clear from context, we will sometimes write just $\radius (A)$ and $\Centers (A)$.
\end{definition}

\begin{corollary}					\label{invariance of centers}
Suppose $A$ is a nonempty closed, invariant subset of $\bd X$.  Then $\Centers (A)$  is a nonempty closed, invariant set. \end{corollary}

\begin{proof}
Invariance is trivial; the rest follows immediately from Lemma \ref{Hausdorff semicontinuity}. \end{proof}

\section{Minimality and Splittings}
		 \label{minimality and splittings}

Following Caprace \cite{caprace-pcmi}, we say $G \le \Isom(X)$ acts \defn{minimally} on $X$ if $G$ does not preserve any proper nonempty, closed convex subset $X' \subsetneq X$, and we say $G \le \Isom(X)$ is \defn{geometrically dense} if $G$ acts minimally on $X$ without fixed point at infinity.
Similarly, call a subset $Y \subseteq X$ \defn{boundary-minimal} if there is no proper, closed convex subset $Z \subsetneq Y$ such that $\bd Z = \bd Y$.

The following splitting theorem is due to Caprace \cite[Proposition III.10]{caprace-pcmi}.

\begin{proposition}					\label{caprace splitting}
Let $X$ be a proper $\textnormal{CAT(0)}$ space on which $\Isom(X)$ acts cocompactly and minimally.
Then $\bd X$ splits as a nontrivial join if and only if $X$ splits as a nontrivial product. \end{proposition}

Of course one would like to know sufficient conditions for minimality.
The following result is due to work of Caprace and Monod \cite{cm-st}.

\begin{proposition}					\label{caprace minimality}
Let $X$ be a proper $\textnormal{CAT(0)}$ space and let $G \le \Isom(X)$.
\begin{enumerate}
\item
If $X$ is geodesically complete and $G$ acts cocompactly then $G$ acts minimally.
\item
If $G$ does not fix a point in $\bd X$ then $G$ stabilizes a nonempty closed convex subspace $X' \subseteq X$ on which its action is geometrically dense.
\item
If $G$ acts cocompactly on $X$ then $G$ stabilizes a quasi-dense, closed convex subspace $X' \subseteq X$ on which its action is minimal.
\end{enumerate}
\end{proposition}

(A subset $Y \subseteq X$ is \defn{quasi-dense} if there is some $R > 0$ such that every $x \in X$ has $d(x, Y) \le R$.
Clearly, every quasi-dense subset $Y$ of $X$ has $\bd Y = \bd X$.)

\section{Proof of \thmrefstar{main radius rigidity} and \correfstar{main radius rigidity corollary}}

For a CAT($1$) space $Y$, let $\dim Y$ denote the geometric dimension of $Y$ (see \cite{kleiner}).  Thus, following our convention, throughout this paper, $\dim \bd X$ will always denote the geometric dimension of the Tits boundary of $X$.

We now prove \thmref{main radius rigidity}.
It will follow from the following lemma.
Note that by Kleiner \cite[Theorem C]{kleiner}, if $\Isom(X)$ acts cocompactly on $X$ then $\dim \bd X$ is finite.

\begin{theorem} [\thmref{main radius rigidity}]		\label{minimal radius rigidity}
Let $X$ be a proper \textnormal{CAT($0$)} space with finite-dimensional Tits boundary, and let $G \le \Isom(X)$ be geometrically dense.
If $\bd X$ has a nonempty $G$-invariant subset of circumradius $\le \frac{\pi}{2}$
then $X$ splits as a nontrivial product. \end{theorem}

\begin{proof}
Assume $A \subset \bd X$ is nonempty, $G$-invariant, and has $\radius^{\bd X} (A) \le \frac{\pi}{2}$.
If we had $\radius^{\bd X} (A) < \frac{\pi}{2}$ then the unique circumcenter of $A$ would be fixed by $G$, so $\radius^{\bd X} (A) = \frac{\pi}{2}$.
By semicontinuity of the angle metric, we may assume $A$ is closed, hence $A \in \CIG$.
Let $C = \Centers^{\bd X} (A)$; by Corollary \ref{invariance of centers}, $C \in \CIG$.

Fix $x \in X$.  For $a \in A$, let $H_a$ be the closed horoball with respect to $a$ such that $x$ is in the boundary horosphere.
Note that $H_a$ is closed and convex.
Moreover, $C \subseteq \cl{B}_T (a, \frac{\pi}{2})$ by definition of $C$, and since
$\bd H_a = \cl{B}_T (a, \frac{\pi}{2})$ \cite[Lemma 3.5]{cm-st}, 
we have $C \subseteq \bd H_a$ for all $a \in A$.

Now let $Y = \bigcap_{a \in A} H_a$.
Then
\[C = \Centers^{\bd X} (A) = \bigcap_{a \in A} \cl{B}_T (a, \frac{\pi}{2}) = \bigcap_{a \in A} \bd H_a \supseteq \bd Y.\]
Also, for each $c \in C$ and $a \in A$, we know that $C \subseteq \bd H_a$, while $x \in Y$, and thus each geodesic ray $[x, c) \subseteq H_a$ by convexity of $H_a$.
Hence $[x,c) \subseteq Y$ for each $c \in C$, and therefore $C \subseteq \bd Y$.
Thus we see that $Y$ is closed, convex, and $\bd Y = C$.

In particular, $C$ is a closed, $\pi$-convex subset of $\bd X$.
Since $C$ is $G$-invariant but contains no fixed points, by Balser and Lytchak \cite[Proposition 1.4]{bl-centers}, the intrinsic circumradius $\radius^{C} (C) > \frac{\pi}{2}$.
So by Caprace and Monod \cite[Proposition 3.6]{cm-st} (cf. \cite[Theorem II.7]{caprace-pcmi}), $C = \bd Y = \bd Z$ for some boundary-minimal closed convex $Z \subset X$; furthermore the union of all such spaces is a closed convex set $Z_0 = Z \times Z'$.
Since $C$ is $G$-invariant, each $g \in G$ has $g Z$ boundary-minimal, closed convex, with $\bd (g Z) = C$.
Thus $G$ preserves the fibers of $Z_0$ and therefore $Z_0$ itself.
Hence $Z_0 = X$ by minimality.
\end{proof}

\begin{corollary} [\corref{main radius rigidity corollary}]	\label{radius rigidity}
Let $X$ be a proper \textnormal{CAT($0$)} space and let $G \le \Isom(X)$ act cocompactly on $X$ without fixed point at infinity.
If $\bd X$ has a nonempty $G$-invariant subset of circumradius $\le \frac{\pi}{2}$
then a quasi-dense, closed convex subspace $X' \subseteq X$ splits as a nontrivial product. \end{corollary}

As noted in the introduction,
if $G \le \Isom(X)$ is a unimodular closed subgroup (for instance, if $G$ is discrete), and $G$ acts cocompactly, then the presence of fixed points in $\bd X$ already indicates a nontrivial Euclidean factor of a quasi-dense, closed convex $X' \subseteq X$ \cite[Theorem M]{cm-fpa}.

\section{Limiting Operations}

\begin{standing hypothesis}
In the remainder this paper, $X$ will be a proper CAT($0$) space, and $\G$ will be a discrete group acting geometrically (that is: properly discontinuously, cocompactly, and by isometries) on $X$. \end{standing hypothesis}

\subsection{Limit Operators}

The following construction comes from Guralnik and Swenson \cite{gs}:  Let $G$ be a discrete group acting on a compact Hausdorff space $Z$.  Denote by $\beta G$ the Stone--\VV{C}ech compactification of $G$.
For each $z \in Z$, extend the orbit map $\rho_z \colon g \mapsto g z$ and, for $\w \in \beta G$, define $T^{\w} \colon Z \to Z$ by $T^{\w} z = (\beta \rho_z)(\w)$.
Thus, for fixed $z \in Z$, the map $\w \mapsto T^{\w} z$ is a continuous map of $\beta G$ into $Z$ (although $z \mapsto T^{\w} z$ may not be continuous for fixed $\w \in \beta G$).

Similarly, $\beta G$ has a semi-group structure:  For each $g \in G$ the left multiplication map $L_g \colon h \mapsto gh$ on $G$ extends to a continuous map $\beta L_g \colon \w \mapsto g \cdot \w$ on $\beta G$.  Now we have jointly continuous multiplication $G \times \beta G \to \beta G$; thinking of this as a left action by $G$ on $\beta G$, we can extend it to a left-continuous multiplication map $\beta G \times \beta G \to \beta G$ (that is, $(\nu, \w) \mapsto \nu \cdot \w$ is continuous in $\nu$ for fixed $\w \in \beta G$).

The inverse map $g \mapsto g^{-1}$ on $G$ extends to a continuous involution $S \colon \beta G \to \beta G$; however, $\w \cdot S\w = 1$ only for $\w \in G$.
On the other hand, $T^{\nu\cdot \w} = T^{\nu} T^{\w}$ for all $\nu, \w \in \beta G$, hence the operators $T^{\w}$ extend the group action by $G$ on $Z$ to a semi-group action by $\beta G$ on $Z$ (by the often discontinuous maps $T^{\w}$).

Now let $\G$ be a discrete group acting properly discontinuously and by isometries on a proper CAT($0$) space $X$.
Since $\cl X = X \cup \bd X$ is compact Hausdorff, the above construction gives us a family of operators $T^{\w}$ on $\cl X$.
Guralnik and Swenson observe that since every $\g \in \G$ acts by isometries on $\bd X$, every $T^{\w}$ is in fact $1$-Lipschitz on $\bd X$ by semicontinuity of $\angle$.

\subsection{Folding}

\begin{definition}
Call a metric space $K$ a \defn{sphere} if it is isometric to a Euclidean sphere of radius one.
Call a sphere $K \subset \bd X$ \defn{round} if $\dim K = \dim \bd X$. \end{definition}

Bennett, Mooney, and Spatzier \cite[Corollary 2.5]{bms} observe that, as a result of work of Papasoglu and Swenson \cite{ps}, the union of round spheres is dense in $\bd X$.  Their proof actually shows the following.

\begin{lemma}						\label{density of round spheres}
Suppose $K \subset \bd X$ is a round sphere.  Then $\G K$ is dense in $\bd X$. \end{lemma}

\begin{definition}
Let $K \subset \bd X$ be a sphere, and suppose $\w \in \BG$ has the property that $T^{\w} (\bd X) = T^{\w} K$ and $T^{\w} \res{K}$ is an isometry.  We say that $\w$ \defn{folds} $K$ onto $T^{\w} K$ (or, $\w$ is a \defn{$K$-folding}), and we call $T^{\w} K$ a \defn{folded} sphere. \end{definition}

It turns out that such spheres do exist.

\begin{theorem} [Theorem A of \cite{gs}]
Let $\G$ be a group acting geometrically on a proper $\textnormal{CAT(0)}$ space $X$.  Then $\bd X$ contains a folded round sphere. \end{theorem}

In fact, every round sphere has a folding.

\subsection{$\pi$-convergence}

\begin{definition}
Let $\w \in \BG \setminus \G$.  Then the \defn{attracting point} $\w^+ = T^{\w}(x) \in \bd X$ and \defn{repelling point} $\w^- = T^{S \w}(x) \in \bd X$ of $\w$ do not depend on $x \in X$. \end{definition}

Papasoglu and Swenson's $\pi$-convergence theorem \cite[Lemma 19]{ps} can be restated in terms of $\BG$:

\begin{theorem} [Theorem 3.3 of \cite{gs}]
Let $\G$ be a group acting geometrically on a proper \textnormal{CAT($0$)} space $X$, and let $\w \in \BG \setminus \G$.  Then
\[d(p, \w^-) + d(T^{\w} p, \w^+) \le \pi\]
for all $p \in \bd X$. \end{theorem}

\begin{remark}
Papasoglu and Swenson actually prove a slightly stronger statement in terms of the Tits metric $\dT$ on $\bd X$,
but in this paper we will focus exclusively on the angle metric $d = \angle$. \end{remark}

Points $p,q \in \bd X$ are said to be \defn{$\G$-dual} if there exists some $\w \in \BG \setminus \G$ such that $\w^+ = p$ and $\w^- = q$.  For $p \in \bd X$, let $\D(p)$ denote the set of points $q \in \bd X$ such that $p,q$ are $\G$-dual.

We will use the following two corollaries of the theorem.

\begin{corollary}					\label{pi-convergence for CI}
Let $A \in \CI$ and suppose $p,q \in \bd X$ are $\G$-dual.  Then
\[d(p, A) + \Hangle(q, A) \le \pi.\]
\end{corollary}

\begin{corollary}					\label{pi-r}
Let $p \in \bd X$ and $A \in \CI$.  Then $d(p, A) \le \pi - \radius^{\bd X} (A)$.
\end{corollary}

\begin{lemma}						\label{dual antipodes}
Let $p, q \in \bd X$ be $\Gamma$-dual, and suppose $d(T^{\w} p, T^{\w} q) \ge \pi$ for some $\w \in \BG$.  Then
\[\Hd(T^{\w} p, A) = \Hd (p, A) = \pi - d(q, A) = \pi - d(T^{\w} q, A)\]
for every $A \in \CI$. \end{lemma}

\begin{proof}
$\Hd(T^{\w} p, A) \le \Hd (p, A) \le \pi - d(q, A) \le \pi - d(T^{\w} q, A) \le \Hd (T^{\w} p, A)$. \end{proof}

\section{Centers Lemma}

\begin{lemma}						\label{long string}
Suppose $A \in \CI$ and $K$ is a folded sphere.  Then for any $p \in K$, there exist $p' \in \bd X$ and $n \in K$ such that
\[\Hd (n, A) \le \pi - d(p', A) \le \pi - d(p, A \cap K).\] \end{lemma}

\begin{proof}
Suppose $\w \in \BG$ folds onto $K$.  Let $p' \in \bd X$ such that $T^{\w} p' = p$.  Let $n' \in \D(p')$ and $n = T^{\w} n'$.  Then
\[\Hd (n, A) \le \Hd(n', A) \le \pi - d(p', A) \le \pi - d(p, T^{\w} A) \le \pi - d(p, A \cap K).\qedhere\] \end{proof}

\begin{definition}
For $p \in K$, write $\A_{K} (n)$ for the antipode of $p$ on $K$. \end{definition}

\begin{lemma} [Centers Lemma]				\label{centers lemma}
Let $A \in \CI$.  For any folded sphere $K \subset \bd X$, the set $K \cap \Centers^{\bd X} (A)$ is a nonempty subset of $\Centers^{K} (A \cap K)$.  Moreover,
\[\radius^{\bd X} (A) = \radius^{K} (A \cap K) = \pi - \max_{K} d(\x, A \cap K) = \pi - \sup_{\bd X} d(\x, A),
\vspace{-4pt}\]
and the supremum is realized. \end{lemma}

\begin{proof}
Note $\radius^{K} (A \cap K) = \pi - \max_{K} d(\x, A \cap K)$ is trivial, while $\radius^{\bd X} (A) \le \pi - \sup_{\bd X} d(\x, A)$ because $\Hd (n, A) + d(p, A) \le \pi$ for all $n \in \D(p)$%
.  We will show $\pi - \sup_{\bd X} d(\x, A) \le \pi - \max_{K} d(\x, A \cap K) = \radius^{\bd X} (A)$.

Choose $p \in K$ to maximize $d(\x, A \cap K)$ on $K$.  Find $p' \in \bd X$ and $n \in K$ by \lemref{long string} such that
\[\Hd (n, A) \le \pi - d(p', A) \le \pi - d(p, A \cap K).\]
Let $q = \A_{K} (n)$; then $\Hd (n, A \cap K) + d(q, A \cap K) = \pi$.  By choice of $p \in K$, $d(q, A \cap K) \le d(p, A \cap K)$.  So $\Hd(n, A \cap K) = \pi - d(q, A \cap K) \ge \pi - d(p, A \cap K)$.  Thus we have
\[\pi - d(p, A \cap K) \le \Hd (n, A \cap K) \le \Hd (n, A) \le \pi - d(p', A) \le \pi - d(p, A \cap K),\]
and we obtain equality throughout:
\[\Hd (n, A \cap K) = \Hd (n, A) = \pi - d(p', A) = \pi - d(p, A \cap K).\]
Hence
\[\pi - \sup_{\bd X} d(\x, A) \le \pi - d(p', A) = \pi - d(p, A \cap K) = \pi - \max_{K} d(\x, A \cap K).\]

On the other hand, if $\w \in \BG$ folds onto $K$, then $T^{\w}(\Centers^{\bd X} (A)) \subseteq K \cap \Centers^{\bd X} (A)$; thus there is some $m \in K \cap \Centers^{\bd X} (A)$.  So
\[\Hd(m, A \cap K) \le \Hd (m, A) \le \Hd(n, A) = \pi -d(p, A \cap K)\]
by choice of $m$.  But for $o = \A_{K} (m)$, we know $\Hd(m, A \cap K) + d(o, A \cap K) = \pi$, so
\[\pi - d(p, A \cap K) \le \pi - d(o, A \cap K) = \Hd(m, A \cap K) \le \Hd(m, A) \le \pi - d(p, A \cap K).\]
Hence we have equality throughout; in particular, $\Hd(m, A) = \pi - d(p, A \cap K)$ gives us $\radius^{\bd X} (A) = \pi - \max_{K} d(\x, A \cap K)$, and we have all the equalities claimed by the lemma.  Thus $\Hd(m, A \cap K) = \radius^{K} (A \cap K)$ shows $K \cap \Centers^{\bd X} (A) \subseteq \Centers^{K} (A \cap K)$, and the proof is complete. \end{proof}

\begin{corollary}					\label{proper subspheres}
Let $A \in \CI$.  If $A \cap K$ lies in a proper subsphere of some folded sphere $K \subset \bd X$, then $\radius^{\bd X} (A) \le \frac{\pi}{2}$.
\end{corollary}

\section{$\pi$-convex Subsets of Spheres}

\begin{lemma}						\label{no large proper convex sets}
No proper, $\pi$-convex subset of a sphere $K$ has circumradius $> \frac{\pi}{2}$. \end{lemma}

\begin{proof}
We proceed by induction on $\dim K$.
The case $\dim K = 0$ is clear, so assume $\dim K \ge 1$ and that the lemma holds for all proper subspheres of $K$.
Let $A$ be a $\pi$-convex subset of $K$, and let $r = \radius (A)$.
Suppose, by way of contradiction, that $r > \frac{\pi}{2}$ but $A \neq K$.
Let $p \in K$ be antipodal to a center of $A$, so $p$ maximizes $d(\x, A)$ on $K$ with $d(p, A) = \pi - r$.
Let $K' = \set{p}^\perp \subset K$, and let $B$ be the set of $q \in K'$ such that the half-open geodesic arc $[p,q)$ contains some point of $A$.
Clearly $B$ is $\pi$-convex by $\pi$-convexity of $A$; $B$ is also proper since $p \notin A$ and nonempty because $r > \frac{\pi}{2}$.
So by the induction hypothesis, $B$ has radius $\le \frac{\pi}{2}$ in $K'$.
Let $q \in B$ be antipodal to a center of $B$.
Then $\angle_p (q, b) \ge \frac{\pi}{2}$ for all $b \in B$.
It follows that a small movement from $p$ in the direction of $q$ will increase $d(\x, A)$, contradicting our choice of $p$.
Thus $A$ cannot have radius $> \frac{\pi}{2}$ without $A = K$.
\end{proof}

\begin{corollary}					\label{pi-convex}
Let $\G$ be a group acting geometrically on a proper \textnormal{CAT($0$)} space $X$.
If $A \in \CI$ $\bd X$ is $\pi$-convex,
then $\bd X$ contains a nonempty, invariant subset of circumradius $\le \frac{\pi}{2}$. \end{corollary}

\begin{proof}
Let $A \in \CI$ be $\pi$-convex and let $r = \radius^{\bd X} (A)$.
Let $K$ be a folded round sphere in $\bd X$; clearly $A \cap K$ is $\pi$-convex.
By \lemref{centers lemma}, $\radius^{K} (A \cap K) = r$.  Also, $r < \pi$:  Otherwise, $A \cap K = K$ implies $A = \bd X$ because $\G K$ is dense in $\bd X$.
But no sphere has a proper, $\pi$-convex subset of circumradius $> \frac{\pi}{2}$, so $r = \radius^{K} (A \cap K)$ must be $\le \frac{\pi}{2}$.
\end{proof}

\begin{lemma}						\label{decompositions on spheres}
Let $A * B$ be a complete \textnormal{CAT($1$)} space containing a round sphere $K \subseteq A * B$, and let $S_A = A \cap K$ and $S_B = B \cap K$.  Then $(S_A)^\perp = B$ and $(S_B)^\perp = A$, and $K = S_A * S_B$. \end{lemma}

\begin{proof}
We first show $(S_A)^\perp = B$.  Clearly $B \subseteq (S_A)^\perp$.  So suppose $q \in A * B \setminus B$; then there is some $p \in A$ such that $d(p, q) < \frac{\pi}{2}$.  By Balser and Lytchak \cite[Lemma 3.1]{bl-centers}, there is some $n \in S_A$ such that $d(p, n) \ge \pi$.  Thus $d(n, q) \ge d(n, p) - d(p, q) > \frac{\pi}{2}$, hence $q \notin (S_A)^\perp$ because $n \in S_A$.  Therefore, $(S_A)^\perp = B$.

Thus $S_B = (S_A)^\perp \cap K$, which is a subsphere of $K$ by linear algebra.  By symmetry, $(S_B)^\perp = A$ and $S_A = (S_B)^\perp \cap K$.  Hence $K = S_A * S_B$. \end{proof}

\section{Conclusion}

Before completing the proof of \thmref{equivalences} from the introduction, we point out for comparison the following known result, which details the case of suspensions:

\begin{proposition}					\label{suspension equivalences}
Let $\G$ be a group acting geometrically on a proper \textnormal{CAT($0$)} space $X$ with $\abs{\bd X} \ge 3$.  The following are equivalent.
\begin{enumerate}
\item
\label{s1}
Some finite index subgroup $\G_0$ of $\G$ fixes a point of $\bd X$.
\item
\label{s2}
For some finite index subgroup $\G_0$ of $\G$, $\bd X$ contains a nonempty, $\G_0$-invariant subset of circumradius $< \frac{\pi}{2}$.
\item
\label{s3}
$\bd X$ splits as a suspension.
\item
\label{s4}
A quasi-dense, closed convex subset $X' \s X$ splits as a product with $\R$.
\end{enumerate}
Moreover, if $X$ is minimal then each of the above is equivalent to
\begin{enumerate}
\setcounter{enumi}{4}
\item
\label{s5}
$X$ splits as a nontrivial product with $\R$.
\end{enumerate}
\end{proposition}

\begin{proof}
\itemrefstar{s3}$\Rightarrow$\itemrefstar{s2}
follows from Foertsch and Lytchak \cite[Theorem 1.1]{fl}
by coning, and
\itemrefstar{s3}$\Leftrightarrow$\itemrefstar{s4}%
$\Leftrightarrow$\itemrefstar{s5}
from Propositions \ref{caprace splitting} and \ref{caprace minimality}.
\itemrefstar{s1}$\Rightarrow$\itemrefstar{s2}
is trivial.
Papasoglu and Swenson showed
\itemrefstar{s2}$\Rightarrow$\itemrefstar{s1}
\cite[proof of Theorem 23]{ps} and also
\itemrefstar{s1}$\Rightarrow$\itemrefstar{s3}
\cite[paragraph following the proof of Lemma 26]{ps}.
\end{proof}

We now complete the proof of \thmref{equivalences} from the introduction.

\begin{theorem} [\thmref{equivalences}]			\label{nonmain equivalences}
Let $\G$ be a group acting geometrically on a proper \textnormal{CAT($0$)} space $X$ with $\abs{\bd X} \ge 3$.
Let
\[\closedsets^{f.i.} = \setp{A \subseteq \bd X}{A \in \CInaught \textnormal{ for some finite index subgroup $\G_0$ of $\G$}}.\]
The following are equivalent.
\begin{enumerate}
\item
\label{e1}
There is some $A \in \closedsets^{f.i.}$ such that
for every round sphere $K \s \bd X$, $A \cap K$ lies on a proper subsphere of $K$.
\item
\label{e2}
There is some $A \in \closedsets^{f.i.}$ of circumradius $\le \frac{\pi}{2}$.
\item
\label{e6}
There exists a proper, $\pi$-convex set in $\closedsets^{f.i.}$.
\item
\label{e3}
$\bd X$ splits as a nontrivial spherical join.
\item
\label{e4}
A quasi-dense, closed convex set $X' \s X$ splits as a nontrivial product.
\end{enumerate}
Moreover, if $X$ is minimal then each of the above is equivalent to
\begin{enumerate}
\setcounter{enumi}{5}
\item
\label{e5}
$X$ splits as a nontrivial product.
\end{enumerate}
\end{theorem}

\begin{proof}
As in the preceding proof, \itemrefstar{e3}$\Leftrightarrow$\itemrefstar{e4}%
$\Leftrightarrow$\itemrefstar{e5}
follows from Propositions \ref{caprace splitting} and \ref{caprace minimality}.
\itemrefstar{e3}$\Rightarrow$\itemrefstar{e1}
follows from \lemref{decompositions on spheres},
\itemrefstar{e1}$\Rightarrow$\itemrefstar{e2}
from \corref{proper subspheres}, and
\itemrefstar{e2}$\Rightarrow$\itemrefstar{e3}
from \thmref{main radius rigidity}.
\itemrefstar{e4}$\Rightarrow$\itemrefstar{e6} is straightforward,
and 
\itemrefstar{e6}$\Rightarrow$\itemrefstar{e2}
follows from \corref{pi-convex}.
\end{proof}

We remark that the subgroup $\G_0$ in parts \itemrefstar{e1}, \itemrefstar{e2}, and \itemrefstar{s3} may be taken to be $\G \cap \Isom_0(X)$, where $\Isom_0(X)$ is the subgroup of $\Isom(X)$ that does not permute the factors of the de Rham decomposition of $X$, and fixes pointwise the boundary of the Euclidean factor.

\bibliographystyle{amsplain}
\bibliography{refs}

\end{document}

%% file: artcfg.tex
\usepackage{amssymb}
\usepackage{hyperref}
\usepackage{tikz}

\usepackage{environ}

\usepackage{import}

\theoremstyle{plain}
\newtheorem{thm}{Theorem}
\newtheorem{theorem}[thm]{Theorem}
\newtheorem*{theorem*}{Theorem}

\newtheorem{corollary}[thm]{Corollary}
\newtheorem{lemma}[thm]{Lemma}
\newtheorem*{lemma*}{Lemma}

\newtheorem*{conjecture*}{Conjecture}

\newtheorem{proposition}[thm]{Proposition}
\newtheorem*{fact*}{Fact}
\newtheorem*{question*}{Question}

\newtheorem{main}{Theorem}

\newtheorem{maintheorem}[main]{Theorem}

\newtheorem{maincorollary}[main]{Corollary}

\theoremstyle{remark}

\newtheorem*{remark}{Remark}

\theoremstyle{definition}
\newtheorem{definition}[thm]{Definition}
\newtheorem*{convention}{Convention}

\newtheorem*{standing hypothesis}{Standing Hypothesis}

\newcommand{\thmref}[1]{Theorem~\ref{#1}}
\newcommand{\thmrefstar}[1]{Theorem~\ref*{#1}}
\newcommand{\corref}[1]{Corollary~\ref{#1}}
\newcommand{\correfstar}[1]{Corollary~\ref*{#1}}
\newcommand{\lemref}[1]{Lemma~\ref{#1}}

\newcommand{\itemrefstar}[1]{(\ref*{#1})}

\newcommand{\defn}[1]{\emph{#1}}

\graphicspath{{pictures/}}

\newcommand{\R}{\mathbb{R}}

\renewcommand{\setminus}{\smallsetminus}

\DeclareMathOperator{\radius}{radius}

\DeclareMathOperator{\Isom}{Isom}

\newcommand{\set}[1]{\left\{#1\right\}}
\newcommand{\setp}[2]{\left\{#1 \mid #2\right\}}
\newcommand{\abs}[1]{\left| #1 \right|}

\newcommand{\res}[1]{\vert_{#1}}

\newcommand{\cl}[1]{\overline{#1}}
\newcommand{\bd}{\partial}

\DeclareMathOperator{\Dualset}{\mathcal D}
\newcommand{\D}{\Dualset}

\DeclareMathOperator{\A}{\mathcal A}

\DeclareMathOperator{\Centers}{Centers}

\newcommand{\w}{\omega}
\newcommand{\VV}[1]{\ensuremath{\check{\mathrm{#1}}}}
\newcommand{\G}{\Gamma}
\newcommand{\BG}{\beta \Gamma}
\newcommand{\Hd}{\mathcal{H}d}
\newcommand{\Hangle}{\Hd}
\newcommand{\g}{\gamma}

\newcommand{\s}{\subseteq}

\newcommand{\x}{-}

\newcommand{\closedsets}{\mathcal C (\bd X)}
\newcommand{\invariantsets}{\closedsets^\G}
\newcommand{\invariantsetsnaught}{\closedsets^{\G_0}}
\newcommand{\invariantsetsG}{\closedsets^G}

\newcommand{\CI}{\invariantsets}
\newcommand{\CInaught}{\invariantsetsnaught}
\newcommand{\CIG}{\invariantsetsG}
\DeclareMathOperator{\dT}{d_T}
\renewcommand{\epsilon}{\varepsilon}